\documentclass[12pt,reqno]{amsart}

\usepackage[T1]{fontenc}
\usepackage{lmodern}
\usepackage[a4paper,margin=1.02in]{geometry}
\usepackage{amsmath,amssymb,amsthm,mathtools}
\usepackage{microtype}
\usepackage{enumitem}
\usepackage[colorlinks=true,linkcolor=blue,citecolor=blue,urlcolor=blue]{hyperref}

\newtheorem{theorem}{Theorem}[section]
\newtheorem{proposition}[theorem]{Proposition}
\newtheorem{lemma}[theorem]{Lemma}

\theoremstyle{definition}

\theoremstyle{remark}
\newtheorem{remark}[theorem]{Remark}

\newcommand{\Q}{\mathbb Q}
\newcommand{\Z}{\mathbb Z}
\newcommand{\R}{\mathbb R}
\newcommand{\Qp}{\mathbb Q_p}
\newcommand{\Zp}{\mathbb Z_p}

\newcommand{\calA}{\mathcal A}
\newcommand{\calB}{\mathcal B}
\newcommand{\calV}{\mathcal V}
\newcommand{\Span}{\operatorname{span}}

\title[On the $p$-adic Wirsing problem]{On the $p$-adic Wirsing problem}
\author{Anup B. Dixit}

\address{Department of Mathematics\\ Institute of Mathematical Sciences (HBNI)\\ CIT Campus, IV Cross Road\\ Chennai\\ India-600113}
\email{anupdixit@imsc.res.in}
\date{\today}

\begin{document}

\begin{abstract}
For a real transcendental number $\xi$, let $\omega_n^*(\xi)$ denote the supremum of all $\omega$ for which there exist infinitely many real algebraic numbers $\alpha$ of degree $\leq n$ satisfying $|\xi-\alpha|\leq H(\alpha)^{-\omega -1}$, where $H(\alpha)$ is the naive height of the minimal polynomial of $\alpha$. A celebrated result of Wirsing gives the uniform lower bound $\omega_n^*(\xi)\geq\frac{n+1}{2}$, which was improved significantly in a recent work of Po\"els to $\frac{n}{2-\log 2}$. In this paper, we establish a $p$-adic counterpart of Po\"els's result. Let $p$ be a prime and $\xi\in\Qp$ be transcendental. Let $\omega_{n,p}^*(\xi)$ be the supremum of all real numbers $\omega$ for which there exist infinitely many algebraic numbers $\alpha \in \Qp$ of degree $\leq n$ such that $|\xi-\alpha|_p\leq H(\alpha)^{-\omega -1}$. We show that $\omega^*_{n,p}(\xi)\geq\frac{n}{2-\log 2}-1$. This improves the known lower bounds in the $p$-adic setting, namely the analogue of Wirsing's theorem, due to Morrison and Teuli\'e.
\end{abstract}


\maketitle

\bigskip
\section{\bf Introduction}
\bigskip
In 1842, Dirichlet \cite{Dirichlet} proved that for every irrational real
number $\xi$, there exist infinitely many rational numbers $a/q$ such that
\begin{equation}\label{eq:dirichlet-approximation}
0<\left|\xi-\frac aq\right|<\frac1{q^2}.
\end{equation}
This exponent $2$ cannot be improved to $2+\varepsilon$ for any
$\varepsilon>0$. More precisely, Roth \cite{Roth} proved that if $\xi$ is an
irrational algebraic number, then
$$
0<\left|\xi-\frac aq\right|<\frac1{q^{2+\varepsilon}}
$$
has only finitely many rational solutions $a/q$. However, the approximation
in \eqref{eq:dirichlet-approximation} can be improved by replacing
rational numbers with algebraic numbers of bounded degree.\\

For a non-zero polynomial $P(X)=a_0+a_1X+\cdots+a_dX^d\in\Z[X]$, its naive height is defined as
\begin{equation*}
    H(P)\coloneqq\max_{0\leq i\leq d}|a_i|.
\end{equation*}
For $\alpha \in \overline{\Q}\setminus \{0\}$, denote by $H(\alpha)$ the height of its minimal polynomial over $\Z$. For a real number $\xi$, let $\omega_n^*(\xi)$ denote the
supremum of all real numbers $\omega$ for which there are infinitely many real algebraic numbers $\alpha$ of degree at most $n$ satisfying
$$
0<|\xi-\alpha|\leq H(\alpha)^{-\omega-1}
$$
In \cite{Wirsing}, Wirsing proved that
$$
\omega_n^*(\xi)\geq\frac{n+1}{2}
$$
for $\xi\in\R$, not algebraic of degree $\leq n$. He conjectured the optimal lower bound to be $\omega_n^*(\xi)\geq n$. Davenport and Schmidt \cite{DavenportSchmidt} proved this conjecture for $n=2$ and it
remains open for every $n\geq3$. A significant improvement to Wirsing's theorem was made by Badziahin
and Schleischitz \cite{BadziahinSchleischitz}, who showed that for a real transcendental number $\xi$, $\omega_n^*(\xi) \geq n/\sqrt3$ for $n\geq4$. Most recently, Po\"els
\cite{PoelsWirsing} made a major breakthrough by introducing a new method that combines parametric geometry of numbers with generalized resultants. He proved that, for real transcendental $\xi$ and $n\geq2$,
\begin{equation}\label{eq:poels-bound}
\omega_n^*(\xi)\geq\frac{n}{2-\log2}.
\end{equation}
For a comprehensive account of the history of Wirsing's problem, we refer the reader to
\cite{BugeaudBook} and the excellent survey \cite{BugeaudSurvey}.\\

In this paper, we consider the $p$-adic analogue of Wirsing's problem. Fix a prime $p$ and normalize the $p$-adic absolute value by $|p|_p=p^{-1}$. Let $\xi\in\Qp$ be transcendental. Define $\omega^*_{n,p}(\xi)$ to be the
supremum of the real numbers $\omega$ for which there are infinitely many numbers $\alpha \in \Qp$ that are algebraic over $\Q$, with
$$
[\Q(\alpha):\Q]\leq n,
$$
satisfying
$$
0<|\xi-\alpha|_p\leq H(\alpha)^{-\omega-1}.
$$

\noindent
A $p$-adic counterpart of Wirsing's theorem was first obtained by Morrison \cite{Morrison} and subsequently developed by Teuli\'e \cite{Teulie}, who proved that
\begin{equation}\label{eq:morrison-bound}
\omega^*_{n,p}(\xi)\geq\frac{n+1}{2}
\end{equation}
whenever $\xi\in\Qp$ is not algebraic of degree at most $n$. Further results on approximation by algebraic numbers in the $p$-adic setting include the study of quadratic approximation exponents by Bugeaud and Pejkovi\'c \cite{BugeaudPejkovic} and a Khintchine-type theorem
for approximation by $p$-adic algebraic numbers due to Beresnevich, Bernik and Kovalevskaya \cite{BeresnevichBernikKovalevskaya}.\\

Our main theorem provides a $p$-adic counterpart to Po\"els's result.

\begin{theorem}\label{thm:main}
Let $p$ be a prime, let $n\geq2$, and let $\xi\in\Qp$ be transcendental. Then
$$
\omega^*_{n,p}(\xi)
\geq\frac{n}{2-\log2}-1.
$$
\end{theorem}

This lower bound is stronger than \eqref{eq:morrison-bound} for every $n\geq6$. The loss of one relative to Po\"els's result \eqref{eq:poels-bound} reflects an additional height factor that appears in the $p$-adic determinant estimate.
Note that the approximation of $\xi$ is by algebraic numbers $\alpha$, which are required to lie in $\Qp$.
The other conjugates of $\alpha$ may lie in arbitrary completions.

\subsection{Strategy of the proof}

We adapt the method developed by Po\"els \cite{PoelsWirsing} in the real
setting. His argument begins with a family of integer polynomials having
controlled heights and small values at the number being approximated. It then
introduces and uses the notion of generalized resultants and a finite-dimensional minimization problem to
produce a polynomial with the estimates needed to locate a nearby real algebraic
number. The underlying successive-minima viewpoint comes from the parametric
geometry of numbers developed by Schmidt and Summerer
\cite{SchmidtSummerer2009,SchmidtSummerer2013} and Roy \cite{Roy}, while the use of
$n+1$ polynomials simultaneously originates in Po\"els's earlier work in
\cite{PoelsUniform}.

The main difficulty in mimicking Po\"els's argument is that there is no
$p$-adic analogue of Minkowski's convex body theorem suited to this
construction. We overcome this difficulty by introducing a family of ``congruence lattices'' adapted to
 polynomial evaluation at $\xi$, which recast the relevant $p$-adic conditions as conditions on an ordinary
lattice to which the classical geometry of numbers applies. We then adapt
Po\"els's generalized-resultant construction and minimization argument to
produce a polynomial satisfying the strong Hensel condition at $\xi$.
Hensel's lemma therefore yields the required algebraic approximation in
$\Qp$.

\bigskip
\section{\bf Preliminaries}
\bigskip

For a non-zero polynomial $P(X)=a_d\prod_{i=1}^d(X-\alpha_i)\in\mathbb C[X]$, its Mahler measure is defined as
$$
M(P):= |a_d|\prod_{i=1}^d\max\{1,|\alpha_i|\}.
$$
We shall use the standard estimates below (see
\cite[Sections~1.6--1.7]{BombieriGubler}).

\begin{lemma}\label{lem:height-mahler}
Let $P\in\Z[X]$ have degree at most $d$. Then
$$
2^{-d}H(P)\leq M(P)\leq\sqrt{d+1}\,H(P).
$$
If $R,S\in\Z[X]\setminus\{0\}$, then
$$
M(RS)=M(R)M(S).
$$
If $R$ divides $S$ in $\Z[X]$ and $\deg S\leq d$, then
$$
H(R)\leq 2^d\sqrt{d+1}\,H(S).
$$
\end{lemma}


To obtain a nearby algebraic number from a polynomial satisfying suitable
$p$-adic estimates, we use the following strong form of Hensel's lemma (see
Gouv\^ea \cite{Gouvea}).

\begin{lemma}\label{lem:hensel}
Let $f\in\Qp[X]$ and $z\in\Qp$. If
$$
|f(z)|_p<|f'(z)|_p^2,
$$
then there exists $\alpha\in\Qp$ with $f(\alpha)=0$ and
$$
|z-\alpha|_p\leq\frac{|f(z)|_p}{|f'(z)|_p}.
$$
\end{lemma}

For $\xi\in \Qp$, define $w_{n,p}(\xi)$ as the supremum of real numbers $w$ for which
\begin{equation}\label{defn:w_{n,p}}
0<|P(\xi)|_p\leq H(P)^{-w-1}
\end{equation}
has infinitely many solutions $P\in\Z[X]$ of degree at most $n$. For a non-zero polynomial $P\in\Z[X]$ define
$$
\Delta_\xi(P)\coloneqq M(P)|P(\xi)|_p.
$$
Since both factors are multiplicative, we have
$$
\Delta_\xi(PQ)=\Delta_\xi(P)\Delta_\xi(Q).
$$
If $M(P)>1$, define
$$
\nu_\xi(P)
\coloneqq-\frac{\log\Delta_\xi(P)}{\log M(P)}
$$
or equivalently,
$$
\Delta_\xi(P)=M(P)^{-\nu_\xi(P)}.
$$

\begin{lemma}\label{lem:factor-lower}
Let $\xi\in\Qp$ be transcendental and $P(X)\in\Z[X]$ be a non-zero polynomial of degree $\leq n$. Let $V\geq 0$ be such that for every non-constant irreducible factor $R$ of $P$ with $M(R)>1$ satisfies
$$
\nu_\xi(R)\leq V.
$$
Then, there exists a constant $\kappa=\kappa(n,p,\xi)>0$, independent of $P$ such that 
$$
\Delta_\xi(P)\geq\kappa M(P)^{-V}.
$$
\end{lemma}

\begin{proof}
Write
$$
P=cU\prod_{j=1}^sR_j^{e_j},
$$
where $c\in\Z\setminus\{0\}$, all irreducible factors of $U$ have Mahler
measure $1$, and $M(R_j)>1$. We know that $\Delta_\xi(c)=|c||c|_p\geq1$ and $M(U)=1$. By Kronecker's theorem, only finitely many irreducible polynomials of degree at most $n$ have Mahler measure $1$. Since $\xi$ is transcendental, we also have $\Delta_\xi(U)\gg_{n,p,\xi}1$. Therefore, by multiplicativity of $\Delta_\xi$, we conclude
$$
\Delta_\xi(P)
\gg_{n,p,\xi}\prod_{j=1}^sM(R_j)^{-e_jV}
\geq M(P)^{-V}.
$$
\end{proof}

The quantity $\nu_\xi(P)$ measures the quality of approximation to $\xi$ by $P$, normalized with respect to the Mahler measure of $P$. Therefore, we define the following invariant.
\begin{equation}\label{eq:tau-definition}
\tau_{n,p}(\xi)
\coloneqq
\limsup_{\substack{P\in\Z[X]_{\leq n},\ M(P)>1\\H(P)\to\infty}}
\nu_\xi(P).
\end{equation}
\begin{lemma}\label{lem:dirichlet-quality}
Let $\xi\in\Qp$ be transcendental. Then
\begin{enumerate}[label=\textup{(\roman*)}]
    \item $$
            \tau_{n,p}(\xi)\geq n.
        $$
        \item The $\limsup$ in \eqref{eq:tau-definition} may be restricted to non-constant irreducible polynomials, i.e.,
        \begin{equation}\label{eq:irreducible-extraction}
            \tau_{n,p}(\xi)
                =
                \limsup_{\substack{P\in\Z[X]_{\leq n} \text{ irred},\\\ M(P)>1,\,\,H(P)\to\infty}}
                \nu_\xi(P).
        \end{equation}
        \item The exponent $\tau_{n,p}(\xi)$ agrees with the polynomial approximation
exponent introduced in \eqref{defn:w_{n,p}}, i.e.,
\begin{equation}\label{eq:tau-w}
\tau_{n,p}(\xi)=w_{n,p}(\xi).
\end{equation}
\end{enumerate}

\end{lemma}

\begin{proof}
Let $H\geq1$. Choose $h$ so that
$$
p^h<(H+1)^{n+1}\leq p^{h+1}.
$$
The $(H+1)^{n+1}$ polynomials with all coefficients in $\{0,1,\ldots,H\}$ take only $p^h$ possible values modulo $p^h\Zp$. Therefore, two of them have the same value. Their difference is a non-zero polynomial $P\in\Z[X]_{\leq n}$ with
$$
H(P)\leq H,
\qquad |P(\xi)|_p\leq p^{-h}<p(H+1)^{-n-1}.
$$
As $H\to\infty$, the heights of the resulting polynomials must tend to infinity; otherwise one fixed non-zero polynomial would have values of arbitrarily large $p$-adic valuation at the transcendental number $\xi$. Since $H(P)\leq H$,
$$
|P(\xi)|_p\ll H(P)^{-n-1}.
$$
Using $M(P)\ll_n H(P)$ gives
$$
\Delta_\xi(P)\ll_n H(P)^{-n},
$$
which proves (i). For (ii), write
$P(X)=cU\prod_{j=1}^sR_j(X)^{e_j}$ and put
$$
V(P)\coloneqq\max_{1\leq j\leq s}\nu_\xi(R_j),
$$
with $V(P)=0$ if $s=0$. By Lemma \ref{lem:factor-lower},
$$
\Delta_\xi(P)\geq\kappa M(P)^{-V(P)}
$$
and hence
$$
\nu_\xi(P)\leq V(P)+\frac{|\log\kappa|}{\log M(P)}.
$$
When $M(P) \to\infty$, the last factor tends to zero. Thus, one of the irreducible factors $R$ of $P$ satisfies $\nu_\xi (R)=\nu_\xi(P)+o(1)$ as $M(P)\to\infty$. If all such irreducible factors belonged to a fixed finite set, then, since $\deg P\leq n$, each of them could occur only with bounded multiplicity. Their total contribution to both $M(P)$ and $\Delta_\xi(P)$ would therefore remain bounded. Any unbounded growth of $M(P)$ would then have to come from the integer factor $c$. But for every non-zero integer $c$ we have $\Delta_\xi(c)=|c||c|_p \geq 1$, so such a factor cannot produce a positive limiting value of $\nu_\xi(P)$. Therefore, the $\limsup$ in \eqref{eq:tau-definition} may be restricted to non-constant irreducible polynomials as in \eqref{eq:irreducible-extraction}.

For (iii), note that for polynomials of degree at most $n$, by Lemma \ref{lem:height-mahler}
$$
\log M(P)=\log H(P)+O_n(1).
$$
Moreover,
$$
\nu_\xi(P)
=
-1-\frac{\log|P(\xi)|_p}{\log M(P)}.
$$
Taking upper limits along sequences of unbounded height and replacing $\log M(P)$ by $\log H(P)$, we deduce \eqref{eq:tau-w}.
\end{proof}

A critical role in the proof is played by a minimization bound. For $x\geq n$, let $\calA(x)$ be the set of
vectors $\mathbf a=(a_2,\ldots,a_{n+1})\in\R^n$ satisfying
$$
1=a_2\leq a_3\leq\cdots\leq a_{n+1},
\qquad \sum_{i=2}^{n+1}a_i=x.
$$
For $2\leq k\leq n+1$, put
$$
A_k(x,\mathbf a)
\coloneqq2x-2(n-k+1)-\sum_{i=2}^k a_i,
$$
$$
F(x,\mathbf a)\coloneqq\max_{2\leq k\leq n+1}
\frac{A_k(x,\mathbf a)}{a_k}
\qquad \text{and}\qquad
F_n\coloneqq\inf_{x\geq n}\ \min_{\mathbf a\in\calA(x)}F(x,\mathbf a).
$$

Po\"els \cite[Theorem~8.1]{PoelsWirsing} established the following lower
bound for $F_n$. His minimization argument is inspired by a method of de la
Vall\'ee Poussin \cite[Chapter~VI]{ValleePoussin}.
\begin{theorem}[Po\"els]\label{thm:optimization}
For every integer $n\geq2$,
$$
F_n\geq\frac{n}{2-\log 2}.
$$
\end{theorem}

\bigskip

\section{\bf Congruence lattices and geometry of numbers}
\bigskip


Let $\xi\in\Zp$ be transcendental. The condition
$|P(\xi)|_p\leq p^{-h}$ is equivalent to a linear congruence on the
coefficient vector of $P$, which we exploit to construct a sublattice of $\Z^{n+1}$.\\ 

For every integer $h\geq0$, define
$$
\Lambda_h(\xi)
\coloneqq
\{P\in\Z[X]_{\leq n}:|P(\xi)|_p\leq p^{-h}\}.
$$
We identify $\Z[X]_{\leq n}$ with $\Z^{n+1}$ through coefficient vectors. Since $\xi\in\Zp$, evaluation at $\xi$ takes integer polynomials to $\Zp$. Consider the map 
$$
\Phi_h:\Z[X]_{\leq n}\longrightarrow\Zp/p^h\Zp,
\qquad
P\longmapsto P(\xi)\bmod p^h\Zp.
$$
This is surjective since constant integer polynomials represent every residue class modulo $p^h$ and
$$
\ker\Phi_h=\Lambda_h(\xi).
$$
Consequently,
$$
\Z[X]_{\leq n}/\Lambda_h(\xi)
\cong
\Zp/p^h\Zp
$$
and
$$
[\Z^{n+1}:\Lambda_h(\xi)]=p^h.
$$
In particular $\Lambda_h(\xi)$ is a full-rank Euclidean lattice in $\R^{n+1}$. Indeed,
$$
p^h\Z^{n+1}\subseteq\Lambda_h(\xi)\subseteq\Z^{n+1}.
$$
Define the unit cube
$$
\mathcal C
\coloneqq
\{P\in\R[X]_{\leq n}:H(P)\leq 1\}.
$$
For $1\leq i\leq n+1$, the $i$-th successive minimum of $\mathcal C$ with respect to $\Lambda_h(\xi)$ is defined as
$$
\lambda_i(h)
\coloneqq
\inf\left\{
\lambda>0:
\dim_{\R}\Span_{\R}\bigl(\Lambda_h(\xi)\cap\lambda\mathcal C\bigr)\geq i
\right\}.
$$
Thus
$$
\lambda_1(h)\leq\lambda_2(h)\leq\cdots\leq\lambda_{n+1}(h).
$$
Equivalently, $\lambda_i(h)$ is the least real number $T$ for which there are $i$ linearly independent polynomials
$$
P_1,\ldots,P_i\in\Z[X]_{\leq n}
$$
satisfying simultaneously
$$
H(P_j)\leq T,
\qquad
|P_j(\xi)|_p\leq p^{-h}
$$
for $1\leq j\leq i$. In particular,
$$
\lambda_1(h)
=
\min_{\substack{0\neq P\in\Z[X]_{\leq n}\\ |P(\xi)|_p\leq p^{-h}}}H(P).
$$

\begin{lemma}\label{lem:lattice-minima}
For every $h\geq0$,
$$
\lambda_1(h)\cdots\lambda_{n+1}(h)\asymp_n p^h.
$$
If $Q\in\Lambda_h(\xi)$ satisfies $H(Q)=\lambda_1(h)$, then $Q$ can be extended to linearly independent polynomials
$$
Q=R_1,R_2,\ldots,R_{n+1}\in\Lambda_h(\xi)
$$
such that
$$
H(R_i)\leq\lambda_i(h)
$$
for every $i$, and therefore
$$
\prod_{i=1}^{n+1}H(R_i)\ll_n p^h.
$$
\end{lemma}

\begin{proof}
The ambient dimension is $d=n+1$ and the coefficient cube has volume $2^d$.
Applying Minkowski's second theorem (see Cassels
\cite[Chapter~VIII]{Cassels}) gives
$$
\frac{p^h}{(n+1)!}
\leq
\lambda_1(h)\cdots\lambda_{n+1}(h)
\leq
p^h.
$$
Since the lattice is discrete and the coefficient cube is compact, the successive minima are attained. To include a prescribed shortest vector, we proceed inductively. Suppose that $Q,R_2,\ldots,R_{i-1}$ have been chosen. By the definition of $\lambda_i(h)$, there are $i$ independent lattice vectors of height at most $\lambda_i(h)$. At least one lies outside the $(i-1)$-dimensional span of the already chosen vectors and we choose it as $R_i$.
\end{proof}

Also, note that $\lambda_1(h)\longrightarrow\infty$ as $h\to\infty$. Otherwise, infinitely many lattices $\Lambda_h(\xi)$ would contain a non-zero polynomial from a fixed finite set. One such polynomial would belong to $\Lambda_h(\xi)$ for arbitrarily large $h$, and hence would vanish at $\xi$, which leads to a contradiction as $\xi$ is transcendental. 

\begin{lemma}\label{lem:shortest-vector}
Let $Y\geq1$, and let $h$ be the largest value with $\lambda_1(h)\leq Y.$ Then
$$
\frac{Y}{p}<\lambda_1(h)\leq Y.
$$
If $Q\in\Lambda_h(\xi)$ satisfies $H(Q)=\lambda_1(h)$, then
$$
|Q(\xi)|_p=p^{-h}.
$$
\end{lemma}

\begin{proof}
Since
$$
p\Lambda_h(\xi)\subseteq\Lambda_{h+1}(\xi),
$$
we have for any $h\geq 0$, $\lambda_1(h+1)\leq p\lambda_1(h)$. By the maximality condition,
$$
Y<\lambda_1(h+1)\leq p\lambda_1(h).
$$
If $|Q(\xi)|_p\leq p^{-h-1}$, then $Q\in\Lambda_{h+1}(\xi)$ and
$$
\lambda_1(h+1)\leq H(Q)\leq Y,
$$
a contradiction.
\end{proof}
\medskip


We now construct an independent family of polynomials containing a prescribed irreducible polynomial and satisfying suitable height bounds. The construction is analogous to that of Po\"els \cite[Lemma~4.1]{PoelsWirsing}. In the present $p$-adic setting, the role of the parametric convex bodies is played by the congruence lattices introduced above.\\

We will use the following elementary observation. For any
$P,R\in\Z[X]$, there exists $m\in\{0,1,2\}$ such that
\begin{equation}\label{eq:bounded-adjustment}
H(R)+2H(P) \geq H(R+mP)\geq H(P).
\end{equation}
Indeed, if $a$ is a coefficient of $P$ with $|a|=H(P)$ and $b$ is the
corresponding coefficient of $R$, then one of $b,b+a,b+2a$ has absolute value
at least $|a|$. The other inequality follows directly from the triangle inequality.\\

Fix a real number $c>0$ satisfying
$$
 c< (2^n\sqrt{n+1})^{-1} =: k_n^{-1}.
$$
Let $P\in\Z[X]_{\leq n}$ be an irreducible polynomial of sufficiently large height. Apply Lemma \ref{lem:shortest-vector} with
$$
Y\coloneqq cH(P).
$$
Let $h$ be the largest value with $\lambda_1(h)\leq Y$ and $Q$ be a shortest vector of $\Lambda_h(\xi)$ i.e., $H(Q)=\lambda_1(h)$. Then
$$
\frac{c}{p}H(P)<H(Q)\leq cH(P),
\qquad |Q(\xi)|_p=p^{-h}.
$$
Note that the polynomials $P$ and $Q$ are coprime. This is because if $P$ divided $Q$, then Lemma \ref{lem:height-mahler} would give
$$
H(P)\leq k_nH(Q)\leq k_ncH(P)<H(P),
$$
leading to a contradiction.\\

\begin{proposition}\label{prop:family}
For every irreducible seed polynomial $P\in\Z[X]_{\leq n}$ of sufficiently
large height, let $Q$ be as above. Then there are linearly independent polynomials
$$
P_1,\ldots,P_{n+1}\in\Z[X]_{\leq n}
$$
with the following properties. Writing $H_i\coloneqq H(P_i)$, we have
$$
P_1=Q,
\qquad P_2=P,
\qquad H_1<H_2\leq\cdots\leq H_{n+1},
$$
the polynomials $P_1$ and $P_2$ are coprime, and
$$
H_1\asymp_{n,p,c}H_2,
\qquad
p^{-h}\ll_{n,p,c}(H_1H_2\cdots H_{n+1})^{-1}.
$$
Finally,
$$
|P_i(\xi)|_p
\leq\max\{p^{-h},|P(\xi)|_p\}
$$
for every $i$.
\end{proposition}

\begin{proof}
Extend $Q=R_1$ to an independent successive-minima family
$$
R_1,R_2,\ldots,R_{n+1}
$$
as in Lemma~\ref{lem:lattice-minima}. Since $P$ and $Q$ are independent, when $P$ is written in the basis $R_1,\ldots,R_{n+1}$, at least one coefficient with index $j\geq2$ is non-zero. Thus
$$
Q,P,R_2,\ldots,\widehat{R_j},\ldots,R_{n+1}
$$
is independent. For each $i\neq1,j$, choose $m_i\in\{0,1,2\}$ as in
\eqref{eq:bounded-adjustment} and put
$$
\widetilde R_i\coloneqq R_i+m_iP.
$$
Adding multiples of $P$ preserves independence. Also,
$$
H(\widetilde R_i)\geq H(P),
$$
while
$$
H(\widetilde R_i)
\leq H(R_i)+2H(P)
\ll_{p,c}H(R_i),
$$
because
$$
H(R_i)\geq\lambda_1(h)>\frac{c}{p}H(P).
$$
After ordering the family
$$
Q,P,(\widetilde R_i)_{i\neq1,j}
$$
by height, $Q$ is the first member and $P$ is the second. Since
$$
H(P)\ll_{p,c}H(R_j),
$$
the replacement of $R_j$ by $P$ and the bounded adjustments change the product of heights by at most a fixed factor. Therefore Lemma~\ref{lem:lattice-minima} gives
$$
\prod_{i=1}^{n+1}H_i\ll_{n,p,c}p^h,
$$
as required. Finally, every $R_i$ belongs to $\Lambda_h(\xi)$, and $|m_i|_p\leq1$. Thus
$$
|R_i(\xi)+m_iP(\xi)|_p
\leq\max\{p^{-h},|P(\xi)|_p\}.
$$
\end{proof}


We now choose seed polynomials for which the associated families have suitably small $p$-adic
values at $\xi$. For a family from
Proposition~\ref{prop:family}, define
$$
x\coloneqq
\frac{\log(H(P_2)H(P_3)\cdots H(P_{n+1}))}{\log H(P_2)}
$$
and
$$
y\coloneqq
\frac{\log(M(P_2)M(P_3)\cdots M(P_{n+1}))}{\log M(P_2)}.
$$
Since $H(P_i)\geq H(P_2)$ for $i\geq2$, we have $x\geq n$.
If $L\coloneqq\log H(P_2)$, the bounded-degree comparison between height and Mahler measure gives
\begin{equation}\label{eq:x-y-comparison}
y=x+O_n\left(\frac{x+1}{L}\right).
\end{equation}
In particular, once $y$ is known to be bounded, the same relation first implies that $x$ is bounded and then gives $y=x+o(1)$ as $L\to\infty$.

\begin{proposition}\label{prop:good-seeds}
Suppose that
$$
\tau_{n,p}(\xi)<\infty.
$$
Then there exists a sequence of irreducible polynomials
$P^{(j)}\in\Z[X]_{\leq n}$ with
$$
H(P^{(j)})\longrightarrow\infty
$$
such that the family attached to $P^{(j)}$ by Proposition~\ref{prop:family} satisfies
$$
\max_{1\leq i\leq n+1}|P_i^{(j)}(\xi)|_p
\ll_{n,p,c,\xi}
H_{1,j}^{-1}H_{2,j}^{-x_j+\varepsilon_j},
$$
where $\varepsilon_j\to0$.
\end{proposition}

\begin{proof}
Put
$$
\tau=\tau_{n,p}(\xi)<\infty.
$$
By \eqref{eq:irreducible-extraction}, choose irreducible polynomials
$P^{(j)}$ of unbounded height such that
$$
\nu_\xi(P^{(j)})=\tau-o(1)\qquad(H_{2,j}\to\infty).
$$
For each $j$, let $P_1^{(j)},\ldots,P_{n+1}^{(j)}$ be the family supplied by
Proposition~\ref{prop:family} for the seed $P^{(j)}$. Thus
$P_2^{(j)}=P^{(j)}$ and, writing $Q_j\coloneqq P_1^{(j)}$, we have
$$
H(Q_j)\asymp H(P^{(j)}),
$$
we also have
$$
M(Q_j)\asymp M(P^{(j)}).
$$
The product estimate in Proposition~\ref{prop:family}, together with the bounded-degree comparisons between height and Mahler measure, gives
\begin{equation}\label{eq:q-upper}
\begin{aligned}
\Delta_\xi(Q_j)
&=M(Q_j)|Q_j(\xi)|_p \ll
\left(\prod_{i=2}^{n+1}M(P_i^{(j)})\right)^{-1}=M(P^{(j)})^{-y_j}.
\end{aligned}
\end{equation}
On the other hand, the definition of $\tau$, applied to $Q_j$, gives
$$
\Delta_\xi(Q_j)\geq M(Q_j)^{-\tau-o(1)}
\qquad(H_{2,j}\to\infty).
$$
Since $M(Q_j)\asymp M(P^{(j)})$, comparison with
\eqref{eq:q-upper} yields
$$
y_j\leq\tau+o(1)\qquad(H_{2,j}\to\infty).
$$
Thus $(y_j)$ is bounded and by \eqref{eq:x-y-comparison}, 
$$
y_j=x_j+o(1)\qquad(H_{2,j}\to\infty).
$$
Consequently,
$$
\nu_\xi(P^{(j)})=\tau-o(1)
\geq y_j-o(1)=x_j-o(1)
\qquad(H_{2,j}\to\infty).
$$
Therefore
$$
\begin{aligned}
|P^{(j)}(\xi)|_p
&=M(P^{(j)})^{-1-\nu_\xi(P^{(j)})}
\leq M(P^{(j)})^{-1-y_j+o(1)}\\
&\ll H_{1,j}^{-1}H_{2,j}^{-x_j+o(1)}
\qquad(H_{2,j}\to\infty).
\end{aligned}
$$
By Proposition~\ref{prop:family}, we have $|Q_j(\xi)|_p=p^{-h_j}$, and the
product estimate in that proposition gives
$$
|Q_j(\xi)|_p\ll H_{1,j}^{-1}H_{2,j}^{-x_j}.
$$
Every other member of the family has $p$-adic value at most the maximum of
$|Q_j(\xi)|_p$ and $|P^{(j)}(\xi)|_p$, so the assertion follows.
\end{proof}

\bigskip
\section{\bf Generalized resultants}
\bigskip

The use of resultants to pass from small polynomial values to algebraic
approximations goes back to Wirsing \cite{Wirsing}. We begin by recalling the
classical construction.

Let $P,Q\in\R[X]$ have degrees $m$ and $\ell$, respectively. Their
resultant $\operatorname{Res}(P,Q)$ is, up to sign, the determinant of the
coefficient vectors of the $m+\ell$ polynomials
$$
P,XP,\ldots,X^{\ell-1}P,
\qquad
Q,XQ,\ldots,X^{m-1}Q.
$$
In particular, $\operatorname{Res}(P,Q)\neq0$ if and only if $P$ and $Q$ are
coprime, or equivalently, if and only if these polynomials form a basis of
$\R[X]_{<m+\ell}$. Po\"els's generalized-resultant construction replaces
$P$ and $Q$ by a finite family of polynomials and selects a basis from their
monomial multiples. We recall this construction and establish the existence of
the generalized-resultant basis required in the $p$-adic setting.\\

 For a
non-empty set $\mathcal S\subseteq\R[X]_{\leq n}$ and an integer $N\geq n$,
define
$$
\calB_N(\mathcal S)
\coloneqq
\{X^rP: P\in\mathcal S\setminus\{0\},\ 0\leq r\leq N-\deg P\},
$$
$$
\calV_N(\mathcal S)
\coloneqq\Span_{\R}\calB_N(\mathcal S),
\qquad
g_{\mathcal S}(N)\coloneqq\dim\calV_N(\mathcal S).
$$


We use the following dimension estimate of Po\"els
\cite[Lemma~5.2]{PoelsWirsing}.

\begin{lemma}\label{lem:intermediate-dimension}
Let $\mathcal S$ consist of $j$ linearly independent polynomials in $\R[X]_{\leq n}$, where $2\leq j\leq n+1$, and suppose their greatest common divisor is $1$. Then, for every $0\leq r\leq n-j+1$,
$$
\dim\calV_{n+r}(\mathcal S)\geq2r+j.
$$
\end{lemma}

The following proposition constructs the generalized-resultant basis needed
for the local determinant argument.

\begin{proposition}\label{prop:generalized-basis}
Let $2\leq k\leq n+1$. Set
$$
N\coloneqq2n-k+1,
$$
and let
$$
P_1,\ldots,P_k\in\Z[X]_{\leq n}
$$
be linearly independent. Suppose that $P_1$ and $P_2$ are coprime and that
$$
H(P_1)\leq\cdots\leq H(P_k).
$$
Then there are $N+1$ integer polynomials
$$
S_0,\ldots,S_N\in\Z[X]_{\leq N}
$$
with the following properties.

\begin{enumerate}[label=\textup{(\roman*)}]
\item The polynomials $S_0,\ldots,S_N$ form a basis of $\R[X]_{\leq N}$.

\item Every $S_j$ has the form
$$
X^tQ_i,
$$
where either $i\geq3$ and $Q_i$ is one of the original polynomials $P_1,\ldots,P_i$, or $i\in\{1,2\}$ and
$$
Q_i\coloneqq(X-\lambda_i)^{n-\deg P_i}P_i
$$
for integers $\lambda_i\in\{1,\ldots,2n+2\}$.

\item If
$$
\max_{1\leq i\leq k}|P_i(\xi)|_p\leq\delta,
$$
then
$$
|S_j(\xi)|_p\leq\delta
$$
for every $j$.

\item The coefficient determinant
$$
D=\det(S_0,\ldots,S_N)
$$
is a non-zero integer and satisfies
$$
|D|_\infty
\ll_n
H_1^{n-k+2}H_2^{n-k+2}H_3\cdots H_k,
$$
where $H_i\coloneqq H(P_i)$.

\item Every root of an $S_j$ is either a root of one of the original
polynomials $P_1,\ldots,P_k$, or belongs to
$\{0,1,\ldots,2n+2\}$.
\end{enumerate}
\end{proposition}

\begin{proof}
Choose integers $\lambda_1,\lambda_2\in\{1,\ldots,2n+2\}$
so that
$$
P_2(\lambda_1)\neq0,
\qquad P_1(\lambda_2)\neq0,
$$
and, if both added exponents are positive, $\lambda_1\neq\lambda_2$. This is possible because each polynomial has at most $n$ roots. Put
$$
Q_i=(X-\lambda_i)^{n-\deg P_i}P_i
\qquad (i=1,2).
$$
Then $Q_1,Q_2$ have degree exactly $n$ and are coprime. Their heights are comparable with $H_1,H_2$, with constants depending only on $n$. The span of
$$
Q_1,Q_2,P_1,\ldots,P_k
$$
has dimension at least $k$. Hence one can select $k-2$ polynomials from
$P_1,\ldots,P_k$ so that, together with $Q_1,Q_2$, they are independent.
Order the selected polynomials according to their position in
$P_1,\ldots,P_k$, and write
$$
Q_j=P_{r_j}\qquad (3\leq j\leq k),
$$
where $r_3<\cdots<r_k$. Since only two members of $P_1,\ldots,P_k$ are
omitted, we have $r_j\leq j$. Therefore
$$
H(Q_j)\leq H_j,
\qquad |Q_j(\xi)|_p\leq\delta.
$$
Let
$$
F\coloneqq\calV_N(Q_1,Q_2).
$$
Since $Q_1,Q_2$ are coprime of degree $n$ and
$$
N-n=n-k+1<n,
$$
the polynomials
$$
Q_1,XQ_1,\ldots,X^{n-k+1}Q_1,
$$
$$
Q_2,XQ_2,\ldots,X^{n-k+1}Q_2
$$
are independent. Thus
$$
\dim F=2(n-k+2).
$$
For each $j=3,\ldots,k$, Lemma~\ref{lem:intermediate-dimension}, applied to the independent family $Q_1,\ldots,Q_j$ at
$$
N=n+(n-k+1),
$$
gives
$$
\dim\calV_N(Q_1,\ldots,Q_j)
\geq2(n-k+1)+j
=\dim F+j-2.
$$
Inductively choose
$$
R_j\in\calB_N(Q_3,\ldots,Q_j)
$$
so that
$$
\dim\bigl(F+\Span(R_3,\ldots,R_j)\bigr)
=\dim F+j-2.
$$
For each $j$, the polynomial $R_j$ is a monomial multiple of some $Q_i$ with $3\leq i\leq j$. Hence
$$
H(R_j)\leq H_j,
\qquad |R_j(\xi)|_p\leq\delta,
$$
because $\xi\in\Zp$. Take  $S_0,\ldots,S_N$ to be the ordered family
$$
Q_1,XQ_1,\ldots,X^{n-k+1}Q_1,
$$
$$
Q_2,XQ_2,\ldots,X^{n-k+1}Q_2,
$$
$$
R_3,\ldots,R_k.
$$
It contains
$$
2(n-k+2)+(k-2)=N+1
$$
independent integer polynomials, and hence is a basis of $\R[X]_{\leq N}$. The determinant is therefore a non-zero integer. The Leibniz formula for the coefficient determinant, together with the height bounds above, gives
$$
|D|_\infty
\ll_n
H_1^{n-k+2}H_2^{n-k+2}H_3\cdots H_k.
$$
The assertion about values follows from the explicit forms of the $S_j$. For the roots, note that a basis polynomial arising from $Q_i$ has the form
$$
X^tQ_i.
$$
If $i=1$ or $2$, then
$$
X^tQ_i
=
X^t(X-\lambda_i)^{n-\deg P_i}P_i,
$$
so its roots are roots of $P_i$ together with $0$ and $\lambda_i$. If
$i\geq3$, then $Q_i$ is one of the original $P_r$, and the only newly
introduced root is $0$. Hence every auxiliary root belongs to
$\{0,1,\ldots,2n+2\}$.
\end{proof}

\begin{remark}\label{rem:raw-determinant}
The exponent of $H_1$ differs from that in Po\"els
\cite[proof of Proposition~5.1]{PoelsWirsing} because, in the real setting,
the estimate for the zeroth divided-Taylor coordinate can be used in the
determinant expansion to save one factor of $H_1$. Here the corresponding
estimate holds only with respect to the $p$-adic absolute value and therefore
does not sharpen the archimedean determinant bound, which retains the factor
$H_1^{n-k+2}$.
\end{remark}


We now combine Po\"els's \cite[proof of Proposition~5.1]{PoelsWirsing}
generalized-resultant basis with a determinant argument at the $p$-adic
place. The basis and its column count come from Po\"els's
\cite[proof of Proposition~5.1]{PoelsWirsing} construction. The archimedean
root-location step in Po\"els's argument is replaced here by the strong form
of Hensel's lemma.\\

For $T\in\Qp[X]$ and $i\geq0$, write
$$
T^{[i]}\coloneqq\frac{1}{i!}T^{(i)}.
$$
If $T\in\Z[X]$, then $T^{[i]}\in\Z[X]$, since its coefficients are integer binomial multiples of the coefficients of $T$. Let
$$
T(X)=a_0+a_1X+\cdots+a_NX^N.
$$
Then
$$
T^{[i]}(\xi)
=
\sum_{r=i}^N\binom{r}{i}a_r\xi^{r-i}.
$$
Hence the linear transformation
$$
(a_0,\ldots,a_N)
\longmapsto
\bigl(T(\xi),T^{[1]}(\xi),\ldots,T^{[N]}(\xi)\bigr)
$$
has matrix
$$
U_\xi
=
\left(u_{i,r}\right)_{0\leq i,r\leq N},
\qquad
u_{i,r}
=
\begin{cases}
\binom{r}{i}\xi^{r-i},&r\geq i,\\
0,&r<i.
\end{cases}
$$
This matrix is upper triangular with every diagonal entry equal to $1$, and therefore
$$
\det U_\xi=1.
$$
Consequently, if $S_0,\ldots,S_N$ are polynomials of degree at most $N$ and $D$ is their coefficient determinant, then
$$
D
=
\det\bigl(S_j^{[i]}(\xi)\bigr)_{0\leq i,j\leq N}.
$$
Indeed, the displayed formula for $T^{[i]}(\xi)$ follows from
$$
\frac{1}{i!}\frac{d^i}{dX^i}X^r
=
\binom{r}{i}X^{r-i}.
$$
The coefficient of $a_i$ in $T^{[i]}(\xi)$ is $1$, while $a_0,\ldots,a_{i-1}$ do not occur, which proves the triangular assertion. If $A$ is the matrix whose columns are the monomial coefficient vectors of the $S_j$, then the matrix of divided Taylor coordinates is $U_\xi A$. Hence
$$
\det(U_\xi A)=\det U_\xi\det A=D.
$$

\begin{proposition}\label{prop:local-resultant}
Let $\xi\in\Zp$ be transcendental and let $2\leq k\leq n+1$ be fixed. For
each $j\geq1$, let
$$
P_{1,j},\ldots,P_{k,j}
$$
be a family satisfying the hypotheses of
Proposition~\ref{prop:generalized-basis}. Put
$$
H_{i,j}\coloneqq H(P_{i,j}),
\qquad
B_{k,j}\coloneqq
H_{1,j}^{n-k+2}H_{2,j}^{n-k+2}H_{3,j}\cdots H_{k,j},
$$
and let $\delta_j>0$ satisfy
$$
\max_{1\leq i\leq k}|P_{i,j}(\xi)|_p\leq\delta_j.
$$
If
$$
\delta_j^3B_{k,j}^2\longrightarrow0
\qquad(j\to\infty),
$$
then, for all sufficiently large $j$, there exist an algebraic number
$$
\alpha_j\in\Qp,
\qquad [\Q(\alpha_j):\Q]\leq n,
$$
and an index $m_j\leq k$ such that
$$
H(\alpha_j)\ll_n H_{m_j,j}
$$
and
$$
|\xi-\alpha_j|_p\ll_{n,p}\delta_j^2B_{k,j}.
$$
\end{proposition}

\begin{proof}
Fixing $j$, we suppress the index $j$ throughout the proof.
Choose the basis $S_0,\ldots,S_N$ of Proposition~\ref{prop:generalized-basis}, and let $D$ be its coefficient determinant. Since $D$ is a non-zero integer, the product formula gives
$$
1=|D|_\infty\prod_\ell|D|_\ell.
$$
For every prime $\ell\neq p$, we have $|D|_\ell\leq1$. Therefore
$$
|D|_p\geq|D|_\infty^{-1}\gg_n B_k^{-1}.
$$
Since the transformation matrix $U_\xi$ is upper triangular with determinant
$1$, we have
$$
D=\det\bigl(S_j^{[i]}(\xi)\bigr)_{0\leq i,j\leq N}.
$$
By construction,
$$
|S_j(\xi)|_p\leq\delta.
$$
For $i\geq2$, the polynomial $S_j^{[i]}$ has integer coefficients and $\xi\in\Zp$, so
$$
|S_j^{[i]}(\xi)|_p\leq1.
$$
Every term in the determinant expansion contains one entry from row $0$ and one entry from row $1$. The ultrametric triangle inequality therefore gives
$$
|D|_p
\leq
\delta\max_{0\leq j\leq N}|S_j'(\xi)|_p.
$$
Consequently, for some $S=S_j$,
$$
|S'(\xi)|_p\gg_n(\delta B_k)^{-1}\,\,\,\, \text{and}\,\,
\qquad |S(\xi)|_p\leq\delta.
$$
It follows that
$$
\frac{|S(\xi)|_p}{|S'(\xi)|_p^2}
\ll_n \delta^3B_k^2.
$$
Since $\delta_j^3B_{k,j}^2\to0$ as $j\to\infty$, for all sufficiently large
$j$ we have
$$
|S(\xi)|_p<|S'(\xi)|_p^2.
$$
By Lemma \ref{lem:hensel}, $S$ has a root $\alpha\in\Qp$ satisfying
$$
|\xi-\alpha|_p
\leq\frac{|S(\xi)|_p}{|S'(\xi)|_p}
\ll_n\delta^2B_k.
$$
As $j\to\infty$, the quantity $\delta^2B_k$ tends to zero because  $\delta^3B_k^2\longrightarrow0$ and $B_k\longrightarrow \infty$.
All auxiliary roots introduced in Proposition~\ref{prop:generalized-basis}
belong to $\{0,1,\ldots,2n+2\}$. Since $\xi$ is transcendental, it does not
belong to this set, and therefore
$$
c_{\xi,n}
\coloneqq
\min_{0\leq a\leq2n+2}|\xi-a|_p
>0.
$$
For all sufficiently large indices $j$, we have
$|\xi-\alpha|_p<c_{\xi,n}$, so $\alpha$ cannot be an auxiliary root. It is
therefore a root of one of the original polynomials $P_m$, with $m\leq k$.
In particular,
$$
[\Q(\alpha):\Q]\leq n.
$$
The primitive minimal polynomial of $\alpha$ divides $P_m$, so Lemma~\ref{lem:height-mahler} gives
$$
H(\alpha)\ll_nH_m.
$$
\end{proof}

\bigskip
\section{\bf Proof of the main theorem}
\bigskip

Let $P_1^{(j)},\ldots,P_{n+1}^{(j)}$ be the sequence of polynomial families
supplied by Proposition~\ref{prop:good-seeds}. Put
$$
H_j\coloneqq H_{2,j},
$$
so that $H_j\to\infty$ as $j\to\infty$. For $1\leq i\leq n+1$, define
$$
a_{i,j}\coloneqq\frac{\log H_{i,j}}{\log H_j}.
$$
Then
$$
a_{2,j}=1\leq a_{3,j}\leq\cdots\leq a_{n+1,j},
$$
$$
a_{2,j}+\cdots+a_{n+1,j}=x_j,
$$
and
$$
a_{1,j}=1+o(1)\qquad(j\to\infty),
$$
because $H_{1,j}\asymp H_{2,j}$. Write
$$
\delta_j
\coloneqq\max_i|P_i^{(j)}(\xi)|_p
\leq H_j^{-a_{1,j}-x_j+\varepsilon_j},
$$
where $\varepsilon_j\to0$ as $j\to\infty$, after absorbing the fixed
multiplicative constant into the exponent. For $2\leq k\leq n+1$, put
$$
b_{k,j}\coloneqq(n-k+2)(a_{1,j}+1)+\sum_{i=3}^ka_{i,j}.
$$
Then
$$
B_{k,j}=H_j^{b_{k,j}}.
$$
Define
$$
D_{k,j}\coloneqq2(a_{1,j}+x_j)-b_{k,j}.
$$
Then
$$
D_{k,j}
=A_k(x_j,\mathbf a_j)+1+(n-k)(1-a_{1,j})
=A_k(x_j,\mathbf a_j)+1+o(1)\qquad(j\to\infty).
$$

\begin{lemma}\label{lem:hensel-margin}
For each $j$, choose $k_j$ such that
$$
\frac{A_{k_j}(x_j,\mathbf a_j)}{a_{k_j,j}}=F(x_j,\mathbf a_j).
$$
Then there is a constant $\rho_n>0$, depending only on $n$, such that
$$
D_{k_j,j}-\frac{a_{1,j}+x_j}{2}\geq\rho_n+o(1)
\qquad(j\to\infty).
$$
Consequently,
$$
\delta_j^3B_{k_j,j}^2\longrightarrow0
\qquad(j\to\infty).
$$
\end{lemma}

\begin{proof}
By Theorem~\ref{thm:optimization} and the choice $k=2$ in the maximum,
$$
F(x,\mathbf a)
\geq
\max\left\{
\frac{n}{2-\log2},
2x-2n+1
\right\}.
$$
The function
$$
g_n(x)
\coloneqq
\max\left\{
\frac{n}{2-\log2},
2x-2n+1
\right\}
-\frac{x-1}{2}
$$
is positive for every $x\geq n$ and tends to infinity with $x$. Its minimum is attained where the two terms are the same and is
$$
\rho_n
\coloneqq
\frac{3n}{4(2-\log2)}-\frac n2+\frac34>0.
$$
Thus
$$
F(x_j,\mathbf a_j)-\frac{x_j-1}{2}\geq\rho_n.
$$
Since $a_{k_j,j}\geq1$ and
$A_{k_j}(x_j,\mathbf a_j)=F(x_j,\mathbf a_j)a_{k_j,j}$,
$$
A_{k_j}(x_j,\mathbf a_j)-\frac{x_j-1}{2}\geq\rho_n.
$$
Using $a_{1,j}=1+o(1)$ as $j\to\infty$ and the formula for $D_{k_j,j}$ gives
$$
D_{k_j,j}-\frac{a_{1,j}+x_j}{2}
\geq\rho_n+o(1)\qquad(j\to\infty).
$$
Finally,
$$
\delta_j^3B_{k_j,j}^2
\leq
H_j^{-3(a_{1,j}+x_j)+2b_{k_j,j}+3\varepsilon_j}
=H_j^{a_{1,j}+x_j-2D_{k_j,j}+3\varepsilon_j}.
$$
The exponent is at most
$$
-2\rho_n+o(1)\qquad(j\to\infty),
$$
and the conclusion follows.
\end{proof}

\begin{proof}[Proof of Theorem~\ref{thm:main}]
We first show that we can assume $\xi\in\Z_p$. For any transcendental $\xi\in\Q_p$, choose an integer $r\geq0$ such that
$$
\zeta\coloneqq p^r\xi\in\Zp.
$$
Suppose that $\beta\in\Qp$ is algebraic with degree at most $n$, and put
$$
\alpha\coloneqq p^{-r}\beta.
$$
Then $[\Q(\alpha):\Q]=[\Q(\beta):\Q]$ and
$$
|\xi-\alpha|_p=p^r|\zeta-\beta|_p.
$$
Moreover, multiplication by the fixed rational number $p^{-r}$ changes the
height of algebraic numbers of degree at most $n$ by at most fixed
multiplicative constants. This follows by substituting $p^rX$ into a
primitive defining polynomial and dividing by its content. Consequently, any
approximation exponent obtained for $\zeta$ is also obtained for $\xi$.
We may therefore assume that $\xi\in\Zp$. 

Put
$$
\tau\coloneqq\tau_{n,p}(\xi)=w_{n,p}(\xi),
$$
where the equality follows from \eqref{eq:tau-w}. By
\cite[Theorem~9.3]{BugeaudBook},
$$
\omega^*_{n,p}(\xi)\leq w_{n,p}(\xi)
\leq\omega^*_{n,p}(\xi)+n-1.
$$
Thus, if $\tau=\infty$, then
$$
\omega^*_{n,p}(\xi)=\infty
$$
and the statement follows. Suppose now that $\tau<\infty$. We shall use the sequence of polynomial families indexed above. For each $j$, choose $k_j$
as in Lemma~\ref{lem:hensel-margin}. Note that, this lemma verifies the limiting
hypothesis of Proposition~\ref{prop:local-resultant}. Hence, for all
sufficiently large $j$, there is an algebraic number
$$
\alpha_j\in\Qp,
\qquad [\Q(\alpha_j):\Q]\leq n,
$$
such that
$$
H(\alpha_j)\ll_nH_j^{a_{k_j,j}}
$$
and
$$
|\xi-\alpha_j|_p
\ll
\delta_j^2B_{k_j,j}
\leq H_j^{-D_{k_j,j}+2\varepsilon_j}.
$$
Moreover,
$$
\frac{D_{k_j,j}}{a_{k_j,j}}
=
\frac{A_{k_j}(x_j,\mathbf a_j)}{a_{k_j,j}}
+
\frac{1+(n-k_j)(1-a_{1,j})}{a_{k_j,j}}
\geq F_n-o(1)\qquad(j\to\infty).
$$
Fix $\eta>0$. For all sufficiently large $j$,
$$
D_{k_j,j}-2\varepsilon_j\geq(F_n-\eta)a_{k_j,j}.
$$
Since $F_n-\eta$ is fixed and
$H(\alpha_j)\ll_nH_j^{a_{k_j,j}}$, we obtain
$$
|\xi-\alpha_j|_p
\ll_{n,p,\xi,\eta}
H(\alpha_j)^{-F_n+\eta}.
$$
Hence, there are infinitely many distinct algebraic numbers $\alpha_j$'s as above. Therefore,
$$
\omega^*_{n,p}(\xi)\geq F_n-1-\eta.
$$
Letting $\eta\to0$ gives
$$
\omega^*_{n,p}(\xi)\geq F_n-1.
$$
We are done since $F_n\geq \frac{n}{2-\log 2}$ by Theorem~\ref{thm:optimization}.
\end{proof}

\bigskip
\section*{\bf Acknowledgements}
\bigskip

The author thanks Anthony Po\"els for sharing the slides from his CIRM talk,
which were helpful in gaining a better understanding of the concept of generalized resultants. ChatGPT (non-pro version) 
was used occasionally as an aid to check details of arguments already developed by the author and in improving the clarity of exposition.

\end{document}